\documentclass[a4paper,12pt,reqno]{amsart}
\usepackage{eurosym}
\usepackage[T1]{fontenc}
\usepackage{amsthm}
\usepackage{amsmath}
\usepackage{amssymb}
\usepackage{mathrsfs}
\usepackage{latexsym}
\usepackage{exscale}
\usepackage{geometry}
\usepackage{graphicx}
\usepackage[utf8]{inputenc}
\usepackage{color}
\usepackage[colorlinks=true, pdfstartview=FitV, linkcolor=blue, citecolor=red, urlcolor=blue]{hyperref}

\definecolor{red}{rgb}{1,0.1,0.1}
\definecolor{blue}{rgb}{0.1,0.1,1}
\definecolor{vb}{RGB}{160,32,240}

\numberwithin{equation}{section}
\newtheorem{theorem}{Theorem}[section]

\newtheorem{definition}[theorem]{Definition}

\newtheorem{lemma}[theorem]{Lemma}

\newtheorem{remark}[theorem]{Remark}

\calclayout
\allowdisplaybreaks
\newcommand{\NN}{\mathbb{N}}
\newcommand{\RR}{\mathbb{R}}

\title{Existence of positive solutions for a semipositone $p(\cdot)$-Laplacian problem}
\author[L. A. Vallejos]{Lucas A. Vallejos}
\address{L. A. Vallejos\\
FaMAF \\
Universidad Nacional de C\'ordoba \\
CIEM (CONICET) \\
5000 C\'ordoba, Argentina}
\email{lucas.vallejos@unc.edu.ar}
\author[R.~E.~Vidal]{Ra\'ul E. Vidal}
\address{R.~E.~Vidal \\
FaMAF \\
Universidad Nacional de C\'ordoba \\
CIEM (CONICET) \\
5000 C\'ordoba, Argentina}
\email{raul.vidal@unc.edu.ar}

\thanks{ The authors are partially supported by CONICET and SECYT-UNC}


\begin{document}

\begin{abstract}
In this paper we find a positive weak solution for a semipositone $p(\cdot )$-
Laplacian problem. More precisely, we find a solution for the problem 
\[
\left\{ 
\begin{array}{cc}
-\Delta _{p(\cdot )}u=f(u)-\lambda  & \text{in }\Omega  \\ 
u>0 & \text{in }\Omega  \\ 
u=0 & \text{on }\partial \Omega 
\end{array}%
\right. ,
\]
 where $\Omega \subset \mathbb{R}^{N}$, $N\geq 2$ is a smooth bounded
domain, $f$ is a contiuous function with subcritical growth, $\lambda >0$
and $\Delta _{p(\cdot )}u=\text{div}(\left\vert \nabla u\right\vert
^{p(\cdot )-2}\nabla u)$. Also, we assume an Ambrosetti-Rabinowitz type of
condition and using the Mountain Pass arguments, comparision principles and regularity principles we prove the existence of positive weak solution for $\lambda $ small enough.
\end{abstract}

\maketitle

\section{Introduction}

The study of variational problems with nonstandard growth conditions is an interesting topic in recent years, even today. $p(x)$-growth conditions can be regarded as an important case of nonstandard $(p,q)$-growth conditions. Some of the results obtained on this kind of problems are for example \cite{F2}, \cite{FZ1}, \cite{FZ2}, \cite{FZ3}, \cite{M}, \cite{Zh1} and \cite{Zh2}. \newline Under $(p,q)$-growth conditions Marcellini \cite{M} proved the existence and regularity of weak solutions of elliptic equations of divergence form with differentiable coefficients in the nondegenerate case. 

Under $p(x)$-growth conditions, Fan and Zhao \cite{FZ1}, \cite{FZ2} and \cite{FZ3} proved that the weak solutions of elliptic equations are H\"{o}lder continuous and its gradients have high
integrability. %\newline

An interesting question is the existence of positive solution for a semipositone $p$-Laplacian and $\phi$-Laplacian problems. For example in \cite{ChS}, \cite{CCSU}, \cite{DGU}, \cite{CdeFL} the authors studied the case of the $p$-Laplacian problem and in \cite{HLS}, the $\phi$-Laplacian problem. In \cite{AdHS} the authors proved the existence of positive weak solution for the semipositone problem through Orlicz-Sobolev spaces. Furthermore, more recent works can be seen at \cite{RF}, \cite{SAM}, \cite{RST}, etc.   

An interesting problem is the existence of weak positive solutions for a semipositone $p(x)$-Laplacian problem. 

In the present work we want find positive weak solutions to the problem

\begin{equation}\label{0.1}
	\left\{ 
	\begin{array}{rclc}
		-\Delta_{p(\cdot)} (u)& = &  f(u)-\lambda,\qquad & \Omega,\\
u&>&0, &\Omega,  \\ 
  u & = & 0, & \partial \Omega .
	\end{array}%
	\right.,  
\end{equation}
where $\Omega\subset \mathbb{R}^N$, $N>2$ is a smooth bounded domain with
smooth boundary denoted by $\partial \Omega $, $f:[0,+\infty )\rightarrow 
[0,+\infty )$ is a continuous function with subcritical growth, $\lambda>0$ and $\Delta _{p(\cdot )}u=\text{div}(\left\vert \nabla u\right\vert ^{p(\cdot
)-2}\nabla u)$ is the $p(\cdot )$-Laplacian operator with $2 \leq p(x)<N$ for all $x\in \Omega$. 

\bigskip 

Related to the function $f$, we assume  the following conditions:

\begin{itemize}
\item[$(f1)$] $0=f(0)=\min_{t\in \lbrack 0,\infty )}f(t)$;

\item[$(f2)$] $\lim_{t\rightarrow 0^{+}}\frac{f(t)}{t^{r(x)-1}}=0$, with $r(\cdot )\in (p(\cdot ),p^{\ast }(\cdot ))$ and $r^{-}> p^{+}$; %$r(\cdot )\in (p(\cdot),p^{\ast }(\cdot ))$ and $r^{-}>p^{+}$;
 
\item[$(f3)$] There exists $q(\cdot )\in (p(\cdot ),p^{\ast }(\cdot ))$, $q^{-}> p^{+}$
such that%
\[
\limsup_{ t \rightarrow \infty }\frac{
f(t) }{ t ^{q(x)-1}}<\infty ;
\]

\item[$(f4)$] \textit{Ambrosetti-Rabinowitz type condition}: There are $\theta >p_{+}$ and $t_{0}>0$ such
that%
\[
\theta F(t)\leq f(t)t\text{, \ }\forall t\geq t_{0}\text{,}
\]
where $F(t)=\int_{0}^{t}f(s)\,ds$.
\end{itemize}

\begin{definition}
    A weak solutions of problem \eqref{0.1} is a positive function $u\in W^{1,p(\cdot)}_0(\Omega)$ such that
    $$
    \int_{\Omega }\left\vert
\nabla u\right\vert ^{p(x)-2}\left\langle \nabla u,\nabla v\right\rangle
dx-\int_{\Omega }f(u)v\,dx+\lambda\int_{\Omega }v\,dx=0 
    $$
    for every $v\in W_0^{1,p(\cdot)}(\Omega)$.
\end{definition}

The aim of this paper is to prove the following result.

%\begin{theorem}\label{principalTheorem}
%   Let $\Omega\subset \mathbb{R}^N$, $N>2$, be a smooth bounded domain with smooth boundary. Let $p(\cdot) \in \mathcal{P}(\Omega)$ such that $2 \leq p(x) < N$ for all $x \in \Omega$. If $p(\cdot)\in C^{0,\gamma}$, and the conditions $(f1), (f2), (f3)$ and $(f4)$ are assumed, then the $p(\cdot)$-Laplacian problem (\ref{0.1}) has a weak positive solution $u \in W_{0}^{1,p(\cdot)}$. 
%\end{theorem}

\begin{theorem}\label{principalTheorem}
Assume $(f1)$-$(f4)$ and $p(\cdot)\in C^{0,\gamma}(\overline{\Omega})$. Then, there exists $\lambda_1>0$  such that if $\lambda\in(0,\lambda_1)$,  the $p(\cdot)$-Laplacian problem (\ref{0.1}) has a weak positive solution $u \in  C^{1,\alpha}(\overline{\Omega})$, for some $\alpha\in(0,1)$. 
\end{theorem}

\bigskip 

\section{Preliminaries}

Let $\mathcal{P}$ be the collection of all measurable functions %$p:\Omega\rightarrow [1,\infty] $ 
$p(\cdot):\Omega\rightarrow [1,\infty) $, with $p^+<\infty$ where, %We define
\begin{eqnarray*}
	p^{-} &=&ess~\inf_{x\in \Omega}~p(x)\text{,} \\
	p^{+} &=&ess~\sup_{x\in \Omega}~p(x)\text{.}
\end{eqnarray*}
 Given $p(\cdot )$, the conjugate
exponent $p^{\prime }(\cdot )$ is defined by%
\begin{equation*}
	\frac{1}{p(x)}+\frac{1}{p^{\prime }(x)}=1\text{.}
\end{equation*}

\begin{definition}
	We say that an exponent $p(\cdot) \in \mathcal{P}$ is $locally \ Log-H%
	\ddot{o}lder \ continuous$, $p(\cdot) \in LH_{0}$, if there exists a
	constant $C_0$ such that for any $x,y \in \Omega$, with $|x-y|<\frac{1}{2}$, 
	\begin{equation*}
		\lvert p(x)-p(y) \rvert< \dfrac{-C_0}{\log (|x-y|)}.
	\end{equation*}
	We say that $p(\cdot) \in \mathcal{P}$ is $Log-H\ddot{o}lder \ continuous
	\ at \ infinity$, $p(\cdot) \in LH_{\infty}$, with respect to a base point $%
	x_0 \in \Omega$ if there exist constants $C_{\infty}$ and $p_{\infty}$ such that
	for every $x \in \Omega$, 
	\begin{equation*}
		\lvert p(x)-p_{\infty} \rvert< \dfrac{C_{\infty}}{\log (e+|x-x_0|)}.
	\end{equation*}
	If $p(\cdot) \in LH_0 \cap LH_{\infty}$ we say that $p(\cdot)$ is $globally \ Log-H\ddot{o}lder \ continuous$ and we write $p(\cdot)\in\mathcal{P}^{log}$.
\end{definition}

\begin{definition}
    Given a exponent $p(\cdot) \in \mathcal{P}$, we say that $p(\cdot)$ is Hölder continuous on $\overline{\Omega}$, which is denoted by $p(\cdot) \in C^{0,\gamma} (\overline{\Omega})$, if there exist a positive constant $L$ and an exponent $\gamma \in (0,1)$ such that
    \begin{equation}
        \lvert p(x)-p(y) \rvert \leq L \lvert x-y \rvert^{\gamma} 
    \end{equation}
    for $x,y \in \overline{\Omega}$.

      We say that a function $u$ is in $C^{1,\gamma}(\overline{\Omega})$ if its gradient $\nabla u \in C^{0,\gamma} (\overline{\Omega})$.
\end{definition}

\begin{remark}
\label{p-regularity}
    Let us note that, as $\Omega$ is bounded, if $p(\cdot) \in LH_0(\Omega)$ %with $p_{+}<\infty$ 
    then $p(\cdot) \in LH_{\infty}$ (see \cite{CF}), hence $p(\cdot) \in \mathcal{P}^{log}$. Also, $p(\cdot) \in C^{0,\gamma} (\overline{\Omega})$ implies that $p(\cdot) \in LH_0(\overline\Omega)$.
\end{remark}

Given $p(\cdot) \in \mathcal{P}$ define the $variable \ Lebesgue \ space$ 
$L^{p(\cdot)}(\Omega)$ as the set of measurable functions $u$ on $\Omega$ for which
the modular
\begin{equation*}
	\rho _{p(\cdot )}(u)=\int_{\Omega}\lvert u(x)\rvert
	^{p(x)}dx \text{,}
\end{equation*}
satisfies $\rho_{p(\cdot )}(u/\lambda )<\infty $ for some $\lambda >0$. 

When the exponent is clear from context, we write simply $\rho_{p(\cdot)}=\rho$. This spaces are the Banach spaces with the norm,
\begin{equation*}
	\lVert u \rVert_{p(\cdot)}=\inf\left\lbrace \lambda>0 :
	\rho(u/\lambda)\leq 1\right\rbrace .
\end{equation*}

We need the following result proved in \cite{C-UFN}
\begin{lemma}\label{norma-modular} 
	Let $p(\cdot)\in \mathcal{P}$ with, then 
	\begin{itemize}
		\item if $\Vert u\Vert _{p(\cdot)}\leq 1,$
		\thinspace \thinspace $\Vert u\Vert _{p(\cdot )}^{p^{+}}\leq \int_{\Omega }|u(x)|^{p(x)}\,dx\leq \Vert
		u\Vert_{p(\cdot )}^{p^{-}}$,
		\item if $\Vert u\Vert _{p(\cdot )}\geq 1,$
		\thinspace \thinspace $\Vert u\Vert _{p(\cdot )}^{p^{-}}\leq \int_{\Omega }|u(x)|^{p(x)}\,dx\leq \Vert
		u\Vert _{p(\cdot )}^{p^{+}}$.
	\end{itemize}
	
	In particular, if $\Vert u\Vert _{p(\cdot )}\leq 1$%
	, 
	\begin{equation*}
		\int_{\Omega}|u(x)|^{p(x)}\,dx\leq \Vert u\Vert _{{p(\cdot )}}.
	\end{equation*}
\end{lemma}

\begin{lemma}[H\"{o}lder's inequality]\label{HI} 
	Let $p(\cdot )\,\in \mathcal{P}$, if $f\in L^{p(\cdot )}(\Omega)$ and $g\in L^{p'(\cdot )}(\Omega)$ then $fg\in L^{1}(\Omega)$ and there
	exists a constant $c>1$ such that 
	\begin{equation*}
		\Vert fg\Vert _{1}\leq c\Vert f\Vert
		_{p(\cdot )}\Vert g\Vert_{p'(\cdot )}.
	\end{equation*}
%{\color{red}If $1<p^-\leq p(x)\leq p^+<\infty$} then $c=(\frac1{p^-}-\frac1{p^+}+1)$.
\end{lemma}

\begin{theorem}\label{topo}
	The topology of the Banach space $L^{p(\cdot)}(\Omega)$ endowed by the norm $\|\cdot\|_{p(\cdot)}$ coincides with the topology of modular convergence. % if and only if $p^+<\infty$.
\end{theorem}

\begin{definition}
	Let $\Omega$ a open set $p(\cdot) \in \mathcal{P}$. The Sobolev space $W^{1,p(\cdot)}(\Omega)$  consist all functions $u\in L^{p(\cdot)}(\Omega)$  with the absolute value of distributional gradient $|\nabla u| \in L^{p(\cdot)}(\Omega)$. 
	$W^{1,p(\cdot)}(\Omega)$ is equipped with the norm 
	$$
	\|u\|_{1,p(\cdot)}=\|u\|_{p(\cdot)}+\|\nabla u\|_{p(\cdot)}.
	$$
	We define the Sobolev space $W^{1,p(\cdot)}_{0}(\Omega)$ by $\overline{C_0^1(\Omega)}$, where the closure is with respect to the norm of $W^{1,p(\cdot)}(\Omega)$.
\end{definition}

\begin{theorem}[The Sobolev embedding, Theorem 8.3.1 \cite{DHHR}]\label{se}
	Let $p \in \mathcal{P}^{log}$ satisfy $1\leq p^-\leq p^+<N$. Then there exists a positive constant $K=K(\Omega,N,p)$, such that, for every $u\in W^{1,p(\cdot)}_0(\Omega)$, we have 
	$$
	\|u\|_{q(\cdot)}\leq K \|u\|_{1,p(\cdot)},
	$$ 
	for any $q(\cdot)\in[p(\cdot),p^*(\cdot)]$, with $p^*(\cdot)=\frac{Np(\cdot)}{N-p(\cdot)}$ and $\|u\|_{1,p(\cdot)}:=\|\nabla u\|_{p(\cdot)}$.
\end{theorem}
% {\color{red} Let us note that, as $\Omega$ is bounded and $p\in C^{0,\gamma}(\overline{\Omega})$, for some $\gamma\in(0,1)$, the last Theorem is true.}

The following theorem covers the results given in Theorem 1.2 of \cite{F1} and Theorem 4.1 of \cite{FZ2}, also see proposition 2.1 of \cite{F}.
\begin{theorem}\label{reg}
    Let $\Omega$ be a a bounded smooth domain in $\RR^N$, $g\in C(\overline{\Omega} \times \RR)$ and $p\in C(\overline{\Omega})$ with $1<p^-\leq p^+<\infty$. Consider the $p(x)$-Laplacian Dirichlet problem
    \begin{equation}\label{0.01}
	\left\{ 
	\begin{array}{rclc}
		-\Delta_{p(\cdot)} (u)(x)& = &  g(x,u),\qquad & \Omega,\\
  u & = & 0, & \partial \Omega .
	\end{array}%
	\right.,  
\end{equation}
\begin{itemize}
    \item[(a)] If $f$ satisfy the sub-critical growth condition:
    $$|g(x,t)|\leq c_1+c_2|t|^{q(x)-1}, \qquad  \text{for all }\, x\in\overline{\Omega} \text{ and } \, t\in \RR,$$
    where $q\in C(\overline{\Omega})$ and $1<q(x)<p^*(x)$, for all $x\in\overline{\Omega}$. Then $u \in L^\infty(\Omega)$ for every weak solution u of \eqref{0.01}.
        \item[(b)] Let $u\in W^{1,p(\cdot)}_0(\Omega)\cap L^\infty(\Omega)$ be a weak solution of \eqref{0.01}. If the exponent $p(\cdot) \in C^{0,\gamma} (\overline{\Omega})$ then $u\in C^{1,\alpha}(\overline{\Omega})$, for some $\alpha \in (0,1)$.
\end{itemize}
\end{theorem}

\begin{definition}
    Let $u$ and $v$ in $W^{1,p(\cdot)}(\Omega)$. We say that $-\Delta_p(\cdot)(u)\leq -\Delta_p(\cdot)(v)$ if for all $\phi\in W^{1,p(\cdot)}_0(\Omega)$, with $\phi\geq 0$
    $$
    \int_{\Omega }\left\vert \nabla u\right\vert ^{p(x)-2}\left\langle \nabla u,\nabla \phi\right\rangle dx \leq \int_{\Omega }\left\vert \nabla v\right\vert ^{p(x)-2}\left\langle \nabla v,\nabla \phi\right\rangle dx.
    $$
\end{definition}
Now we give a comparison principle as follows.
\begin{theorem}[See \cite{FZZ}, Lemma 2.2; \cite{F}, Proposition 2.3]\label{com}  Let $u$ and $v$ in $W^{1,p(\cdot)}(\Omega)$. If $-\Delta_p(\cdot)(u)\leq -\Delta_p(\cdot)(v)$ and $u\leq v$ on $\partial\Omega$, (i.e. $(u-v)^+\in W^{1,p(\cdot)}_0(\Omega)$), then $u\leq v$ in $\Omega$.    
\end{theorem}

\

\section{Proof of the main result }

\bigskip 

In order to proof of Theorem \ref{principalTheorem} we define $f_{\lambda}$ as follow,

\begin{equation*}
f_{\lambda}(t)=	\left\{ 
	\begin{array}{rrl}
		f(t)-\lambda,& \quad\qquad t>0,&\\
-\lambda(t+1),&  -1\leq t\leq 0,&\\ 
  0,& t<-1 .&
	\end{array}%
	\right.
\end{equation*}

We prove the existence of a positive solution for the following auxiliary problem:
\begin{equation}\label{0.2}
	\left\{ 
	\begin{array}{rclc}
		-\Delta_{p(\cdot)} (u)& = &  f_\lambda(u)\qquad & \Omega,\\
u&>&0, &\Omega,  \\ 
  u & = & 0, & \partial \Omega .
	\end{array}%
	\right.,  
\end{equation}
because a solution of \eqref{0.2} is also a solution of \eqref{0.1}. Associated to \eqref{0.2}, we have
the energy functional  $E_{\lambda}:W^{1,p(\cdot)}(\Omega)\to\RR$ defined by,
\[
E_{\lambda}(u)=\int_{\Omega }\frac{\left\vert \nabla u\right\vert ^{p(x)}}{p(x)}\,dx-\int_{\Omega }F_{\lambda}(u)\,dx\text{,}
\]
where $F_{\lambda}(t)=\int_{0}^{t}f_{\lambda}(t)(s)\,ds$. This functional is Fr\'echet differentiable, the derivative $E_{\lambda}^{\prime }:W_{0}^{1,p(\cdot )}(\Omega)\rightarrow \left(
W_{0}^{1,p(\cdot )}(\Omega)\right) ^{\ast }$, where for $u\in W_{0}^{1,p(\cdot)}(\Omega)$, $
E_{\lambda}^{\prime }(u):W_{0}^{1,p(\cdot )}(\Omega)\rightarrow \mathbb{R}$ is given by
\begin{equation}\label{der}
\left\langle E_{\lambda}^{\prime }(u),v\right\rangle =\int_{\Omega }\left\vert
\nabla u\right\vert ^{p(x)-2}\left\langle \nabla u,\nabla v\right\rangle
dx-\int_{\Omega }f_{\lambda}(u)v\,dx\text{.}
\end{equation}

The critical points $u$ of $E_{\lambda}$, i.e.,

\begin{equation}\label{CriticalPoint}
    \left\langle E_{\lambda}^{\prime }(u),v\right\rangle = 0,
\end{equation}
for all $v \in W_0^{1,p(\cdot)}$, are weak solutions of \eqref{0.2}.
So next we need only consider the existence of nontrivial critical points of $E_{\lambda}$. 

We observe that, by $(f1)$ and $(f2)$, there exists a constant $c_1>0$, such that  
\begin{equation}\label{2}
f(t)\leq t^{r(x)-1}+c_{1}t^{q(x)-1}\text{.}   
\end{equation}
%where $c_1=\frac{p^-}{2p^+K^{p^+}}$. 
If $t>0$,%
$$
F_{\lambda}(t)\leq\int_{0}^{t}f(s)ds \leq \frac{1}{r(x)}t^{r(x)}+\frac{c_{1}}{q(x)}%
t^{q(x)}\text{.}    
$$
If $t<0$, the area of triangle is less than $\frac{\lambda}{2}$. Therefore, for
all $t\in \mathbb{R}$ we must have that,%
\begin{align}\label{Fa}
F_{\lambda}(t)\leq \frac{1}{r(x)}t^{r(x)}+\frac{c_{1}}{q(x)}t^{q(x)}+\frac{\lambda}{2}\text{.}
\end{align}

The next lemmas will be useful to prove that $E_\lambda$ verifies the mountain pass geometry.

\begin{lemma}\label{lemma1}
    There exists $0<\delta_1<1$, such that if $\delta \in(0,\delta_1)$ and $\|u\|_{1,p(\cdot)}= \delta$, then there exist  
    $\lambda_1=\lambda_1(p(\cdot),q(\cdot),\delta)$ and $\Lambda_1=\Lambda_1(p(\cdot),q(\cdot),\delta,\lambda_1)$, such that $E_\lambda(u)>\Lambda_1$ for all $\lambda \in (0,\lambda_1)$.   Moreover, the constant $\delta_1$ and $\delta$ are independent of $\lambda \in (0,\lambda_1)$.
\end{lemma} 

\begin{proof}
By Remark \ref{p-regularity} and Theorem \ref{se}, 
if $\|u\|_{1,p(\cdot)}\leq 1$ then
\begin{align*}
     \left\|\frac{u}{K}\right\|_{q(\cdot)}<1,
\end{align*}
for all $q(x)\in[p(x),p^*(x)]$.
By inequality \eqref{Fa}  and Lemma \ref{norma-modular},
\begin{eqnarray*}
E_{\lambda }(u) &=&\int_{\Omega }\frac{\left\vert \nabla u(x)\right\vert
^{p(x)}}{p(x)}dx-\int_{\Omega }F_{\lambda }(u(x))dx \\
&\geq &\frac{1}{p^{+}}\left\Vert u\right\Vert _{1,p(\cdot )}^{p^{+}}-\frac{
1}{r^{-}}\int_{\Omega }u(x)^{r(x)}dx-\frac{c_{1}}{q^{-}}\int_{\Omega
}u(x)^{q(x)}dx-\frac{\lambda }{2}\left\vert \Omega \right\vert  \\
&\geq &\frac{1}{p^{+}}\left\Vert u\right\Vert _{1,p(\cdot )}^{p^{+}}-\frac{%
1}{r^{-}}K^{r^{+}}\int_{\Omega }\left( \frac{u(x)}{K}\right) ^{r(x)}dx-%
\frac{c_{1}}{q^{-}}K^{q^{+}}\int_{\Omega }\left( \frac{u(x)}{K}\right)
^{q(x)}dx-\frac{\lambda }{2}\left\vert \Omega \right\vert  \\
&\geq &\frac{1}{p^{+}}\left\Vert u\right\Vert _{1,p(\cdot )}^{p^{+}}-\frac{1}{r^{-}}K^{r^{+}}\left\Vert \frac{u}{K}\right\Vert _{r(\cdot )}^{r^{-}}-%
\frac{c_{1}}{q^{-}}K^{r^{+}}\left\Vert \frac{u}{K}\right\Vert _{q(\cdot
)}^{q^{-}}-\frac{\lambda }{2}\left\vert \Omega \right\vert  \\
&\geq &\frac{1}{p^{+}}\left\Vert u\right\Vert _{1,p(\cdot
)}^{p^{+}}-C\left\Vert u\right\Vert _{1,p(\cdot )}^{p_0}-\frac{\lambda }{%
2}\left\vert \Omega \right\vert \text{,}
\end{eqnarray*}
where $p_{0} =\min \{r^{-},q^{-}\}>p^{+}$. So we can choose $\delta_{1}$ small
enough, and $\lambda_{1}$, such that if $\delta\in (0,\delta_1)$ and $\lambda \in (0,\lambda_1)$,
\begin{equation}\label{eq1}
   E_{\lambda }(u)\geq \Lambda_1 (p(\cdot ),p_0 ,\delta_{1},\lambda _{1})>0 .
\end{equation}
\end{proof}

\begin{lemma}\label{lemma2}
  There exists $v\in W^{1,p(\cdot)}(\Omega)$ such that $\|v\|_{1,p(\cdot)}>1$ and $E_{\lambda}(v)<0$, for all $\lambda \in(0,\lambda_1)$.
\end{lemma}
\begin{proof}
We take $\varphi \in C_{0}^{\infty }(\Omega)$, $\varphi >0$, with $\left\Vert
\varphi \right\Vert _{1,p(\cdot )}=1$. For all $t>1$, 
\begin{eqnarray*}
E_{\lambda}(t\varphi ) &=&\int_{\Omega }\frac{\left\vert \nabla t\varphi \right\vert ^{p(x)}}{p(x)}
dx-\int_{\Omega }F(t\varphi )dx+\lambda\int_{\Omega }t\varphi ~dx\text{.}
\end{eqnarray*}
By $(f4)$ the function $F$ verifies the inequality $\theta y\leq y't$, for all $t\geq t_0$, then there exist positive constants $A$ and $B$ such that  

\[
F(t)\geq
At^{\theta }+B\text{.}
\]
for all $t>0$.

Now, let considered $t>1$,  
\begin{eqnarray}\label{eq2}
E_{\lambda}(t\varphi ) &=&\int_{\Omega }\frac{\left\vert \nabla t\varphi \right\vert ^{p(x)}}{p(x)}
dx-\int_{\Omega }F(t\varphi )\,dx+\lambda\int_{\Omega }t\varphi \,dx\\ \nonumber
&\leq &\frac{t^{p_{+}}}{p_{-}}\int_{\Omega }\left\vert \nabla \varphi
\right\vert ^{p(x)}dx-At^{\theta }\int_{\Omega }\varphi ^{\theta
}dx-B\left\vert \Omega \right\vert +t\lambda_1\left\Vert \varphi \right\Vert _{1} \\ \nonumber
&\leq&\frac{t^{p_{+}}}{p_{-}}\left\Vert \varphi \right\Vert _{1,p(\cdot
)}-At^{\theta }\left\Vert \varphi \right\Vert^{\theta } _{\theta
}-B\left\vert \Omega \right\vert +t\lambda_1\left\Vert \varphi \right\Vert _{1} \\ \nonumber
&=&\frac{t^{p_{+}}}{p_{-}}-At^{\theta }\left\Vert \varphi \right\Vert^{\theta
} _{\theta }-B\left\vert \Omega \right\vert +t\lambda_1\left\Vert \varphi
\right\Vert _{1}. \nonumber
\end{eqnarray}
So, since $\theta >p_{+}$, we can take $t_1>t_{0}$, with $t_1=t_1(p(\cdot),\theta,f,\phi,\lambda_1,\Omega)$, such that 
\[
E_{\lambda}(t_1\varphi )<0,
\]
   finally take $v=t_1\varphi$.
\end{proof}

\begin{lemma}\label{sol}
	  There exist $\Lambda_2>0$ such that, for each $\lambda \in (0,\lambda _{1})$ the functional $E_{\lambda }$ has a critical point $u_{\lambda}$ of mountain pass type that satisfies $\Lambda_1\leq E_\lambda(u_\lambda) \leq\Lambda_2$, where $\Lambda_1$ and $\Lambda_2$ do not depend of $\lambda$. 
\end{lemma}

\begin{proof}
	We show that $E_{\lambda }$  satisfies the Palais–Smale condition. 	
	Assume that $\{u_{n}\}_{n\in \NN}$ is a sequence in $W_{0}^{1,p(\cdot )}(\Omega)$ such that $\{E_{\lambda }(u_{n})\}$ is bounded and $E_{\lambda }^{\prime }(u_{n})\rightarrow 0$. 
	
 First we proved that the sequence $\{u_{n}\}_{n\in \NN}$ is bounded in $W_{0}^{1,p(\cdot )}$. As $E_{\lambda }^{\prime }(u_{n})\rightarrow 0$ there exist $m>0$ such that 
	\[
	\left\vert \left\langle E_{\lambda }^{\prime }(u_{n}),u_{n}\right\rangle \right\vert
	\leq \left\Vert 
	u_{n}\right\Vert _{1,p(\cdot )}\text{,}
	\]
	for	 $n>m.$ Then, 
	\begin{eqnarray*}
		\left\Vert u_{n}\right\Vert _{1,p(\cdot )} &\geq &-\left\langle
		E_{\lambda }^{\prime }(u_{n}),u_{n}\right\rangle  \\
		%&=&-\int_{\Omega }\left\vert \nabla u_{n}\right\vert ^{p(x)-2}\left\langle\nabla u_{n},\nabla u_{n}\right\rangle~dx+\int_{\Omega}f_\lambda(u_{n})u_{n}~dx \\
		&=&-\int_{\Omega }\left\vert \nabla u_{n}\right\vert ^{p(x)}dx+\int_{\Omega } f_\lambda(u_{n})u_{n}~dx\text{.}
	\end{eqnarray*}
Therefore
	\begin{align}\label{3.1}
	-\left\Vert u_{n}\right\Vert _{1,p(\cdot )}-\int_{\Omega }\left\vert \nabla u_{n}\right\vert ^{p(x)}dx\leq - \int_\Omega f_\lambda(u_{n})u_{n}~dx.
	\end{align}

	Let $C_1$ be a constant such that $\left\vert E_{\lambda }(u_{n})\right\vert \leq C_1$
	for all $n$. From $(f4)$ the function $f_\lambda$ satisfy the next Ambrosetti-Rabinowitz condition
 $$
 \theta F_\lambda(t)\leq f_\lambda(t)t+M \qquad \text{for all}  \quad t\in\RR, 
 $$
%{\color{red} for all $t>0$
%\begin{align*}
 %   \theta F_\lambda(t)&=\theta F(t) -\theta \lambda t\\
 %   &\leq (f(t)-\lambda)t +(1-\theta)\lambda t\\
%    &\leq f_\lambda(t)t,
%\end{align*}
%for all $-1<t<0$
%\begin{align*}
%\theta F_\lambda(t)&=\theta\int_0^t -\lambda(s+1)~ds\\
%&=-\theta\lambda(\frac{1}{2}t+1) t\\
%&= -\lambda(t+1)t+\lambda(t+1)t-\theta\lambda(\frac{1}{2}t+1) t\\
%&=f_\lambda(t) +(1-\frac{\theta}{2})t^2+(1-\theta)\lambda t\\
%&\leq f_\lambda(t) + (\theta-1)\lambda_1 
%\end{align*}
%for all $t<-1$
%\begin{align*}
%\theta F_\lambda(t)&=\theta\int_0^{-1} -\lambda(s+1)~ds\\
%&=f_\lambda(t)+ \theta \frac{\lambda}{2} \\
%&\leq f_\lambda(t)+ \theta \frac{\lambda_1}{2} 
%\end{align*}
%}
where %the $\theta>p^+$ and 
$\theta$ and $M$ are
independent of $\lambda\in(0,\lambda_1)$.
  Then,  we have
%	\begin{align*}
%		&\frac{1}{p_{+}}\int_{\Omega }\left\vert 
%         \nabla u_{n}\right\vert ^{p(x)}dx-\frac{1}{\theta }\int_{\Omega   }f_\lambda(u_{n})u_{n}~dx \\
%             		&\qquad\leq \int_{\Omega }\frac{\left\vert \nabla u_{n}\right\vert ^{p(x)}}{p(x)}dx- \int_{\Omega }F_\lambda(u_{n})dx \\
%		&\qquad\leq C,
%	\end{align*}

	\begin{align}\label{3.2}
	&\frac{1}{p_{+}}\int_{\Omega }\left\vert \nabla u_{n}\right\vert ^{p(x)}dx-
	\frac{1}{\theta }\int_{\Omega }f_\lambda(u_{n})u_{n}~dx\\ \nonumber
 &\qquad\leq \int_{\Omega }\frac{\left\vert \nabla u_{n}\right\vert ^{p(x)}}{p(x)}dx- \int_{\Omega }F_\lambda(u_{n})dx+\frac{1}{
		\theta }M\left\vert \Omega \right\vert\\ \nonumber
 &\qquad \leq C_1+\frac{1}{
		\theta }M\left\vert \Omega \right\vert. \nonumber 
	\end{align}

	Therefore, by \eqref{3.1} and \eqref{3.2} we get 
\begin{align*}
	\left(\frac{1}{p_{+}}-\frac1{\theta}\right)\int_{\Omega }\left\vert \nabla u_{n}\right\vert ^{p(x)}dx-\frac1{\theta}\left\Vert u_{n}\right\Vert _{1,p(\cdot )}\leq C_1+\frac{1}{\theta }M\left\vert \Omega \right\vert, 
\end{align*}
we suppose $\left\Vert u_{n}\right\Vert _{1,p(\cdot )}>1$, by Lemma \ref{norma-modular},
\begin{align}\label{3.3}
\left(\frac{1}{p_{+}}-\frac1{\theta}\right)\left\Vert u_{n}\right\Vert _{1,p(\cdot )}^{p-}-\frac1{\theta}\left\Vert u_{n}\right\Vert _{1,p(\cdot )}\leq C_1-\frac{M }{\theta }\left\vert \Omega \right\vert, 	
\end{align}	
	
	The equation \eqref{3.3} is a bounded "polynomial" of degree $p^-$ in $\left\Vert  u_{n}\right\Vert
	_{1,p(\cdot )}$, then $\left\Vert u_{n}\right\Vert_{1,p(\cdot )}$ is bounded for all $n\in\NN$.  Let $C_2>0$ be a constant such that $\left\Vert u_{n}\right\Vert_{1,p(\cdot )}\leq C_2$.

%Let $\tilde C$ be a constant such that $\| u_{n}\|_{1,p(\cdot)} \leq \tilde C$. 

%Since $\{u_n\}$ is bounded in $W^{1,p(\cdot)}(\Omega)$, 
Without loss of generality, we
may assume that $\{u_n\}$ converges weakly in $W_0^{1,p(\cdot)}(\Omega)$. Since $q(\cdot)< p^*(\cdot)$,  by Remark \ref{p-regularity} and Theorem \ref{se} we can assume that $\{u_n\}$ converges $L^{q(\cdot)}(\Omega)$. 

By $(f3)$ there exist positive constants $K_1$ and $K_2$ such that
$$
|f_\lambda(u_n)|\leq K_1|u_n|^{q(\cdot)-1}+K_2,
$$
where $K_2$ depends on $\lambda_1$,
then   $\{f_\lambda(u_n)\}$ is bounded in $L^{q'(\cdot)}(\Omega)$. Therefore, there exists a constant $C_3>0$ such that $\left\| f_\lambda(u_n)\right\|_{q'(\cdot)} \leq C_3$, for all $n$.

These assumptions and H\"older's inequality Lemma \ref{HI}, imply
\begin{align}\label{3.10}
\nonumber	 \int_{\Omega}\left|  f_\lambda(u_n)(u_n-u_m)\right| \,dx&\leq	 \left\| f_\lambda(u_n)\right\|_{q'(\cdot)}\left\| u_n-u_m\right\|_{q(\cdot)}\\  
	& \leq C_3  \left\| u_n-u_m\right\|_{q(\cdot)}\xrightarrow[n,m \to \infty]{}   0.
\end{align}

%We divide the domain $\Omega$ in two parts:
%$$
%\Omega_1=\{x\in\Omega: p(x)<2\} \qquad \text{ and } \qquad \Omega_2=\{x\in\Omega: p(x)\geq2\}.
%$$	
 From \eqref{der} and \eqref{3.10} it is easy to get
\begin{align}\label{3.4}
\nonumber	&\int_{\Omega}\left\langle |\nabla u_n|^{p(x)-2}\nabla u_n-|\nabla u_m|^{p(x)-2}\nabla u_m, \nabla u_n-\nabla u_m \right\rangle \,dx\\
\nonumber	&\qquad\leq \left| \left\langle E'_\lambda(u_n),u_n-u_m \right\rangle\right|  +\left| \int_{\Omega} f_\lambda(u_n)(u_n-u_m)\,dx\right| \\
\nonumber	&\qquad\qquad+ \left| \left\langle E'_\lambda(u_m),u_n-u_m \right\rangle\right|  +\left| \int_{\Omega} f_\lambda(u_m)(u_n-u_m)\,dx\right|\\
\nonumber	& \qquad\leq 2 C_2\left(  \|E'_\lambda(u_m)\|_{-1,p'(\cdot)}+\|E'_\lambda(u_n)\|_{-1,p'(\cdot)}\right)\\
	 &\qquad\qquad+ \left| \int_{\Omega} f_\lambda(u_n)(u_n-u_m)\,dx\right|+\left| \int_{\Omega} f_\lambda(u_m)(u_n-u_m)\,dx\right| \xrightarrow[n,m \to \infty]{}   0.
\end{align}

So, by \eqref{3.4}, we have
\begin{align}\label{3.8}
\nonumber &\int_{\Omega}|\nabla u_n-\nabla u_m|^{p(x)}\,dx\\
 &\qquad \leq C \int_{\Omega}\left\langle |\nabla u_n|^{p(x)-2}\nabla u_n-|\nabla u_m|^{p(x)-2}\nabla u_m, \nabla u_n-\nabla u_m \right\rangle \,dx\xrightarrow[n,m \to \infty]{}   0.
\end{align}
Therefore, by Theorem \ref{topo} and \eqref{3.8}, we get 
$$
\|\nabla u_n-\nabla u_m\|_{p(\cdot)}\xrightarrow[n,m \to \infty]{}   0,
$$
then $\{u_n\}$ is a Cauchy sequence in $W^{1,p(\cdot)}_0(\Omega)$. This
proves that $E_\lambda$ satisfies the Palais–Smale condition. 
So, by Lemmas \ref{lemma1} and \ref{lemma2}, we can apply the mountain pass theorem (see \cite{W}), the functional $E_{\lambda}$ has a nontrivial critical point $u_{\lambda}\in W_0^{1,p(\cdot)}(\Omega)$. Furthermore, this critical point
is characterized by
\begin{align*}
E_\lambda(u_\lambda)=\min_{\gamma\in\Gamma}\max_{t\in[0,1]}E(\gamma(t)),
\end{align*}
where $\Gamma$ is the set of continuous pathways $\gamma:[0,1]\to W_0^{1,p(\cdot)}(\Omega)$, with $\gamma(0)=0$ and $\gamma(1)=v$. Moreover, from equations \eqref{eq1} and \eqref{eq2} we get
\begin{align}\label{eq3}
    \Lambda_1 \leq E_\lambda(u_\lambda) \leq \frac{t_1^{p_{+}}}{p_{-}}+t\lambda_1\left\Vert \varphi
\right\Vert _{1}:=\Lambda_2.
\end{align}
\end{proof}
\begin{remark}\label{r1}
The solutions $u_\lambda$, with $\lambda\in (0,\lambda_1)$ are uniform bounded in $W_0^{1,p(\cdot)}(\Omega)$. 
\end{remark}
In fact, since $u_\lambda$ is a critical point of $E_\lambda$, then
$$\int_{\Omega}|\nabla u_\lambda|^{p(x)}\,dx=\int_\Omega f_\lambda(u_\lambda)u_\lambda\,dx.$$
From the Ambrosetti-Rabinowitz condition and Lemma \ref{sol}
\begin{align*}
    \left(\frac{1}{p^+}-\frac{1}{\theta}\right)\int_{\Omega}|\nabla u_\lambda|^{p(x)}\,dx&=\frac{1}{p^+}\int_{\Omega}|\nabla u_\lambda|^{p(x)}\,dx-\frac1{\theta} \int_\Omega f_\lambda(u_\lambda)u_\lambda\,dx\\
    &\leq \int_{\Omega}\frac{|\nabla u_\lambda|^{p(x)}}{p(x)}\,dx-\int_\Omega F_\lambda(u_\lambda)+\frac{M}{\theta}|\Omega|\\
    &=E_\lambda(u_\lambda)+\frac{M}{\theta}|\Omega|\leq \Lambda_2+\frac{M}{\theta}|\Omega|,
\end{align*}
then there exist a constant $C_4>0$, that don not depend of $\lambda$, such that
$$\|u_\lambda\|_{1,p(\cdot)}\leq C_4.$$
\begin{remark}\label{r2}
The solutions $u_\lambda$, with $\lambda\in (0,\lambda_1)$ are uniform bounded in $L^\infty(\Omega)$. 
\end{remark}
In fact,    the function $f$ satisfy the sub-critical growth condition, equation \eqref{2},   by item $(a)$ in Theorem \ref{reg} and since $u_\lambda$ are uniform bounded in $W_0^{1,p(\cdot)}(\Omega)$, we get $u_\lambda$ are uniform bounded in $L^\infty(\Omega)$. 

\begin{remark}\label{r3}
By Remark \ref{r1}, Remark \ref{r2} and item $(b)$ in Theorem \ref{reg}, we have $u_\lambda$, with $\lambda\in (0,\lambda_1)$, are uniform bounded in $C^{1,\alpha}(\overline{\Omega})$, for some $\alpha\in(0,1)$.
\end{remark}

%\[
%c^{p_{+}-1}\Delta _{p(\cdot )}(u)\leq \Delta _{p(\cdot )}(cu)\leq
%c^{p_{-}-1}\Delta _{p(\cdot )}(u)\text{.} 
%\]

\begin{proof}[Proof of the Theorem \ref{principalTheorem}] We now prove theorem \ref{principalTheorem} by contradiction. Suppose there exists a
sequence $\{\lambda _{j}\}_{j}$, $\lambda_1>\lambda _{j}>\lambda_{j+1}>0$, for all $j$, such that $\lambda_j$  converging to $0$ and  the measure $m\left( \{x\in \Omega :u_{\lambda _{j}}(x)\leq 0\}\right)
>0$. Let $u_{j}=u_{\lambda _{j}}$.
%We know that the sequence $\{u_{j}\}_{j}$ is bounded on $W^{1,p(\cdot )}(\Omega )$.
Then, by %Remark \ref{r1} and
Remark \ref{r3}
%theorem 1.2 in \cite{F1}, 
the sequence $\{u_{j}\}_{j}$ is uniformly bounded %in $W^{1,p(\cdot)}(\Omega)$ and also 
in $C^{1,\alpha }(\overline{\Omega})$ for some $\alpha \in (0,1)$. Hence, for any $\beta \in
(0,\alpha )$, the sequence $\{u_{j}\}_{j}$ has a subsequence that converges
in $C^{1,\beta }(\overline{\Omega})$ and weakly converges in $W^{1,p(\cdot)}_0(\Omega)$. Let us denote its limit by $u$. Next we prove that $%
u(x)\geq 0$.

%Let $v_{0}\in W_{0}^{1,p(\cdot )}(\Omega )$ be a solution of%
%\begin{equation*}\label{0.1}
%	\left\{ 
%	\begin{array}{rclc}
%	-\Delta _{p(\cdot )}(v_{0})&=&-1,\qquad&\text{in }\Omega \\ 
%v_{0}&=&0,&\text{on }\partial \Omega%
%	\end{array}%
%	\right.,  
%\end{equation*}

Let $v_{j}\in W_{0}^{1,p(\cdot )}(\Omega )$ be a solution of%
\begin{equation*}
	\left\{ 
	\begin{array}{rclc}
	-\Delta _{p(\cdot )}(v_{j})&=&-\lambda_j,\qquad&\text{in }\Omega \\ 
v_{j}&=&0,&\text{on }\partial \Omega%
	\end{array}%
	\right.,  
\end{equation*}
by Theorem \ref{com}, we get $v_j\leq v_{j+1}< 0$, for all $j$. Furthermore if $\phi \in W^{1,p(\cdot)}(\Omega)$, we get
\begin{align*}
    \int_{\Omega }\left\vert
\nabla u_j\right\vert ^{p(x)-2}\left\langle \nabla u_j,\nabla \phi\right\rangle dx&=\int_\Omega(f(u_j)-\lambda_j)\phi dx\\
&>-\int_\Omega\lambda_j\phi dx=\int_{\Omega }\left\vert
\nabla v_j\right\vert^{p(x)-2}\left\langle \nabla v_j,\nabla \phi\right\rangle dx,
\end{align*}
then $u_j\geq v_j$. So the fact that $v_{j}\rightarrow 0$ as $j\rightarrow 0$ implies that $u(x)\geq 0$ for all 
$x\in \Omega $.

%Let $k_{j}=-\lambda _{j}$. And let $v_{j}$ be subsolution of

%\[
%\left\{ 
%\begin{array}{c}
%-\Delta _{p(\cdot )}(v_{j})\leq k_{j}\text{, in }\Omega \\ 
%v_{j}=0\text{ on }\partial \Omega%
%\end{array}%
%\right. \text{.} 
%\]

%Then $v_{j}=(-k_{j})^{\frac{1}{p_{+}-1}}v_{0}$. Indeed,

%\[
%-\Delta _{p(\cdot )}(v_{j})=-\Delta _{p(\cdot )}\left( (-k_{j})^{\frac{1}{%
%p_{+}-1}}v_{0}\right) \leq \left( (-k_{j})^{\frac{1}{p_{+}-1}}\right)
%^{p_{+}-1}\left[ -\Delta _{p(\cdot )}(v_{0})\right] =-k_{j}\text{.} 
%\]

%Since $f(u_{j})-\lambda _{j}\geq k_{j}$, then for the Comparision principle in \cite{F}, we have $u_{j}\geq v_{j}$. So the fact that v_{j}(x)\rightarrow 0$ as $j\rightarrow 0$ implies that $u(x)\geq 0$ for all $x\in \Omega $.

Since $\{f(u_{j})\}_{j}$ is bounded on $L^{q^{\prime }(\cdot )}(\Omega )$,
then $f(u_{j})$ converges weakly. Let $z\in L^{q^{\prime }(\cdot )}(\Omega )$
be the weak limit of such a sequence. We have $z\geq 0$ and if $\phi \in
C_{c}^{\infty }(\Omega )$, then%
\begin{eqnarray*}
\int_{\Omega }\left\vert \nabla u\right\vert ^{p(x)-2}\left\langle \nabla
u,\nabla \phi \right\rangle dx &=&\lim_{j\rightarrow \infty }\int_{\Omega
}\left\vert \nabla u_{j}\right\vert ^{p(x)-2}\left\langle \nabla
u_{j},\nabla \phi \right\rangle dx \\
&=&\lim_{j\rightarrow \infty }~\int_{\Omega }(f(u_{j})-\lambda_j)\phi dx \\
&=&\int_{\Omega }z\phi \text{.}
\end{eqnarray*}

Therefore, $-\Delta _{p(\cdot )}(u)=z$. By Theorems 1.1 and 1.2 in \cite{Z} $u>0$ in $\Omega $ and,
\[
\frac{\partial u}{\partial \gamma }(x)>0\text{, for all }x\in \partial
\Omega \text{,}
\]
where $\gamma $ is the inward unit normal vector of $\partial \Omega $ on $x$. Therefore, since $\{u_{j}\}_{j}$ converges in $C^{1,\beta }(\overline{\Omega})$ to $u$, for $j$ sufficiently large, $u_{j}(x)>0$ for all $x\in \Omega $. But this contradicts the assumption that 
\[
m\left( \{x\in \Omega :u_{\lambda _{j}}(x)\leq 0\}\right) >0\text{.}
\]
This contradiction proves theorem (\ref{principalTheorem}).
\end{proof}
%{\color{red} ver un poco esta ultima frase}

\end{document}